\documentclass[10pt,draftcls,onecolumn]{IEEEtran}
\usepackage{epsfig,amssymb,latexsym,color,pifont,colordvi,multicol}
\usepackage[cmex10]{amsmath}
\usepackage{moreverb}
\usepackage{mathptmx,times,verbatim}
\usepackage{subfig}

\definecolor{backgrey}{rgb}{0.86,0.86,0.86}
\definecolor{dblue}{rgb}{0,0.0,0.5}
\definecolor{dred}{rgb}{0.4,0.2,0}
\definecolor{dgreen}{rgb}{0.0,0.5,0}

\newcommand{\captionfonts}{\small}
\makeatletter  % Allow the use of @ in command names
\long\def\@makecaption#1#2{%
  \vskip\abovecaptionskip
  \sbox\@tempboxa{{\captionfonts #1: #2}}%
  \ifdim \wd\@tempboxa >\hsize
    {\captionfonts #1: #2\par}
  \else
    \hbox to\hsize{\hfil\box\@tempboxa\hfil}%
  \fi
  \vskip\belowcaptionskip}
\makeatother   % Cancel the effect of \makeatletter
\newtheorem{theorem}{Theorem}

\newtheorem{assumption}[theorem]{Assumption}

\newtheorem{remark}[theorem]{Remark}

\newtheorem{definition}[theorem]{Definition}

\newtheorem{lemma}[theorem]{Lemma}

\newenvironment{proof}[1][Proof]{\textbf{#1.} }{\ \hspace*{\fill} \rule{0.5em}{0.5em}}

\title{Stabilization of LTV systems over uncertain channels}

\author{\quad Amit Diwadkar \quad Umesh Vaidya %\quad Arvind U. Raghunathan %
\thanks{A. Diwadkar is a post-doctoral researcher with the Department of Electrical and
Computer Engineering, Iowa State University, Ames, IA, 50011
diwadkar@iastate.edu}
\thanks{U. Vaidya is with the Department of Electrical and
Computer Engineering,
Iowa State University,
Ames, IA, 50011 ugvaidya@iastate.edu}%
}
\begin{document}
\maketitle \thispagestyle{empty} \pagestyle{empty}

\begin{abstract}
In this paper, we study the problem of control of discrete-time linear time varying systems over uncertain channels. The uncertainty in the channels is modeled as a stochastic random variable. We use exponential mean square stability of the closed-loop system as a stability criterion. We show that fundamental limitations arise for the mean square exponential stabilization for the closed-loop system expressed in terms of statistics of channel uncertainty and the positive Lyapunov exponent of the open-loop uncontrolled system. Our results generalize the existing results known in the case of linear time invariant systems, where Lyapunov exponents are shown to emerge as the generalization of eigenvalues from linear time invariant systems to linear time varying systems. Simulation results are presented to verify the main results of this paper.
\end{abstract}


\section{Introduction}
\label{intro}
There has been increased research activity in the area of network controlled systems \cite{antsaklis}. One of the important problems addressed in the area of network controlled systems is that of characterizing the performance limitations on control and estimation caused by unreliable communication channels. In this paper, we continue this line of research to prove limitations results for the stabilization of linear time varying (LTV) systems with uncertain communication channels between the plant and the controller.

There is a long list of literature on control of a system over unreliable communication channels. The problem has been looked at in the context of uncertainty threshold principle in \cite{athans_ku, koning}. The notion of anytime capacity was introduced to study the limitations introduced in control of a system over unreliable communication links \cite{sahai_anytime}. Information theoretic results on communication constraints, due to packet loss, have been addressed in \cite{communication_channel, Elia_mitter, tatikonda_mitter}. In \cite{sino_TCP, Gupta2007439}, the problem of optimal control of linear time invariant (LTI) systems over packet-drop links is studied. The combined problem of estimation and control over unreliable links using two different protocols, User Datagram Protocol(UDP) and Transmission Control Protocol(TCP), is studied with LTI plant dynamics in \cite{Imer20061429, sastry_sino_schven_poolla}. Robust control framework is used in the analysis and synthesis of controllers for Multi Input Multi Output (LTI) systems over unreliable channels in \cite{Elia2005237}. In \cite{Seiler_sengupta} the authors consider packet loss between sensor and controller, and pose the control design problem as an $\rm{H}_{\infty}$ optimization problem. Similarly, Markov jump linear systems results are also used in \cite{JLS_aut_cont, costa_LQR_MJLS} to study the network problem over packet-drop links. The problem of characterizing limitations for stabilization and observation of nonlinear systems over erasure channels is also studied and appeared in \cite{vaidya_amit_journal,vaidya_elia_erasure, Vaidya_scl}

In this paper, we study the problem of feedback control of an LTV plant in the presence of uncertainty in communication link connecting the plant and controller. The uncertainty in the communication channels is modeled as a stochastic random variable $\gamma$ with mean $\mu$ and variance $\sigma^2$ (refer to Figure \ref{fig:sch}). The main results of this paper prove that fundamental limitations arise for mean square stabilization of LTV systems over uncertain channels. The limitations are expressed in terms of variance of the uncertainy $\sigma^2$, the mean of the uncertainty $\mu$, and the instability of the open-loop plant dynamics. The instability of the open-loop LTV plant dynamics is captured using positive {\it Lyapunov exponents}. Roughly speaking, Lyapunov exponents can be thought of as the generalization of eigenvalues from LTI systems to LTV systems, and are used to characterize exponential stability/instability of LTV systems. With this analogy, it is interesting to compare the results obtained in this paper with the existing results on the control of LTI systems over uncertain channels \cite{sastry_sino_schven_poolla, Imer20061429, Elia2005237, networksystems_martin_gupta}. The conditions obtained in \cite{sastry_sino_schven_poolla, Imer20061429, Elia2005237, networksystems_martin_gupta} are derived for the stability criterion of bounded mean square stability. We employ a stronger notion of stability given by exponential mean square stability, that characterizes exponential decay in the mean square sense. In particular, under appropriate assumptions, we prove that for a system with $N$ inputs (where $N$ is the dimension of the state space), limitation is a function of only the maximum Lyapunov exponent; whereas, in the single input case, all positive Lyapunov exponents play a role to determine the limitations for control. We provide some discussion on the differences between the results obtained for the $N$ inputs and single input case in Section \ref{ctrlr}, Remark \ref{remark_NSingle}. The results obtained in this paper are consistent with the existing results for LTI systems, where Lyapunov exponents emerge as a natural generalization of eigenvalues from LTI to LTV systems. There are two main contributions of the paper. First, we developed a framework employing tools from ergodic theory of dynamical systems and control theory to study the problem in network controlled systems with LTV dynamics. Second, we provided computable analytical conditions for the stability of feedback controlled LTV systems with uncertainty in the actuation channels. Connection between system entropy and Lyapunov exponents \cite{minimum-entropy, vaidya_amit_journal} can also be used to provide alternate information-theoretic interpretation of our limitation results.

The organization of this paper is as follows. In section \ref{prelim}, we provide some preliminaries and state the main assumptions on the system dynamics. The main results of this paper are proven in section \ref{ctrlr}. Simulation results are presented in section \ref{sim}, followed by conclusions in section \ref{inf}.
\section{Preliminaries}
\label{prelim}
We consider the problem of control of multi-state multi-input LTV systems with a stochastic memoryless multiplicative uncertainty between the plant and the controller (refer to Fig. \ref{fig:sch}a). The LTV system with multiplicative uncertainty channel is described by the following equation:
\begin{align}
\label{ctrl_eqn}
x(t+1) &= A(t)x(t) + \gamma(t)B(t)u(t),
\end{align}
where $x(t) \in \mathbb{R}^N$ is the state, $u(t) \in \mathbb{R}^M$ is input with $M\leq N$, and $t\geq 0$. The channel uncertainty, between the plant and the controller, is modeled using the random variable $\gamma(t)$ and is assumed to satisfy following statistics, $E[\gamma(t)]=\mu$ and $E[(\gamma(t)-\mu)^2]=\sigma^2$. By defining a new random variable, $\Delta(t):=\gamma(t)-\mu$, the feedback control system in Fig. \ref{fig:sch}a can be redrawn as shown in Fig. \ref{fig:sch}b. The random variable $\Delta(t)$ now satisfies
\begin{align}
\label{delta_uncertainty}
E\left[\Delta\right] = 0, \quad \quad   E\left[\Delta^2\right] = \sigma^2,
\end{align}
The feedback control system inside the dotted line in Fig. \ref{fig:sch}b now represents a nominal system with mean connectivity $\mu$, interacting with zero mean $\Delta$ uncertainty with variance $\sigma^2$. The system Eq. (\ref{ctrl_eqn}) can be written as:
\begin{figure}[h!]
\centering
%\label{fig:sch1}
\subfloat[]{\label{fig:sch1}\includegraphics[width=0.4\textwidth]{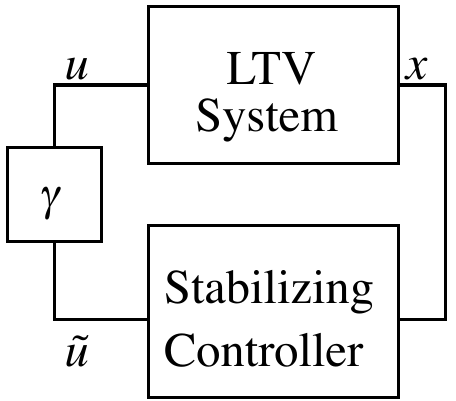}}
\hspace{0.4in}
%\label{fig:sch2}
\subfloat[]{\label{fig:sch2}\includegraphics[width=0.43\textwidth]{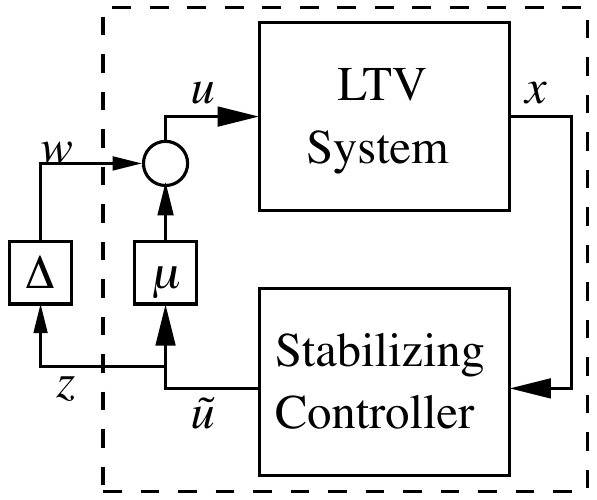}}
\caption{(a) A schematic of LTV system with multiplicative uncertainty in actuation, (b) A schematic of nominal LTV system interconnected with zero mean random variable.}
\label{fig:sch}
\end{figure}
\begin{align}
\label{ctrl_eqn1}
x(t+1) &= A(t)x(t) + \mu B(t)u(t) + \Delta(t)B(t)u(t).
\end{align}
Writing the feedback control system with multiplicative channel uncertainty, $\gamma$, as the interconnection of nominal system with mean connectivity, $\mu$, and zero mean random variable, $\Delta$, closely follows \cite{Elia2005237}.
\begin{remark} The block diagram in Fig. \ref{fig:sch}b, where a nominal system (inside the dotted box) is interconnected with zero mean random variable $\Delta$ allows us to interpret the main results of this paper along the lines of the results known in the robust control literature for LTI systems \cite{paganini_dull, Elia2005237} (refer to Remark \ref{remark_robust}).
\end{remark}
We make the following assumptions about the system dynamics.
\begin{assumption}\label{unif_cont} We assume the system matrix, $A(t)$, is uniformly bounded above and below, and that $B'(t)B(t)$ is uniformly bounded from below. Furthermore, we assume the pair, $(A(t), B(t))$, is uniformly controllable. The definition of uniform controllability is from \cite{kwakernaak} and is given as follows.
\end{assumption}
\begin{definition}[Uniformly controllable]\label{unif_contrb}
The sequence of pairs, $(A(t),B(t))$, is said to be uniformly controllable, if there exists an integer $k \geq 1$ and positive constants $\alpha_0$, $\alpha_1$, $\beta_0$ and $\beta_1$, such that
\begin{align}
&W(t_0,t_0+k) > 0\\
&\alpha_0I \leq W^{-1}(t_0,t_0+k) \leq \alpha_1I \\
&\beta_0 I \leq \Phi'(t_0+k,t_0)W^{-1}(t_0,t_0+k)\Phi(t_0+k,t_0) \leq \beta_1 I,
\end{align}
$\forall t_0$, where $W(t_0,t_1)$ is a symmetric nonnegative matrix.
\begin{align}
W(t_0,t_1) = \sum_{t=t_0}^{t_1-1}\Phi(t_1,t+1)B(t)B'(t)\Phi'(t_1,t+1)
\end{align}
and $\Phi(t,t_0) = \prod_{k=t_0}^{t}A(k)$ is the transition matrix.
\end{definition}
We now provide the following definition of exponential stable and exponentially antistable dynamics for the LTV system. The following two definitions closely follow \cite{minimum-entropy}.
\begin{definition}[Exponential stable and antistable \cite{minimum-entropy}]\label{exp_stab_unstab}
Consider the uncontrolled system in (\ref{ctrl_eqn1}) given by $x(t+1) = A(t)x(t)$. Let $k$, $l$ be positive integers. We say that $\{A(t)\}_{t\geq 0}$ is
\begin{enumerate}
\item {\it Uniformly exponentially stable}, if there exist positive constants $K_s$ and $\beta_s < 1$, such that $\Big\|\prod_{t=k}^{k+l-1}A(t)\Big\| < K_s\beta_s^l$.
\item {\it Uniformly exponentially antistable}, if there exist positive constants $K_u$ and $\beta_u > 1$, such that $\underline{\mu}\left(\prod_{t=k}^{k+l-1}A(t)\right) > K_u\beta_u^l$, where for any $N\times N$ matrix $M$, $\underline{\mu}(M) := inf\{\|Mx\|\colon \|x\| =1\}$.
\end{enumerate}
\end{definition}

For limitation results involving LTI systems, it is known that the fundamental limitations for stabilization using state feedback controller arise only due to antistable parts of the system \cite{Elia2005237, networksystems_foundation_sastry}. A first step towards proving such results for the LTI system is to perform a change of coordinates that allows one to decompose the system matrix into stable and antistable components. We expect similar conclusions to hold true for the limitations results involving LTV systems. In fact, using results from \cite{exponential_dich, minimum-entropy}, it can be shown that the LTV system admits decomposition into stable and antistable components under the assumption that system matrices $\{A(t)\}$ satisfy exponential dichotomy property, defined as follows.
\begin{definition}[Exponential Dichotomy \cite{minimum-entropy}]\label{exp_dichotomy}
Let $\{A(t)\}$ be a sequence of $N\times N$ matrices and let $P(t)$ be a bounded sequence of projections in $\mathbb{R}^N$, such that the rank of $P(t)$ is constant. The sequence $\{P(t)\}$ is a dichotomy for $\{A(t)\}$ if the commutativity condition, $A(t)P(t) = P(t+1)A(t)$, is satisfied for all $t$, and there exists positive constants $L_d$ and $\beta_d > 1$, such that
\begin{align*}
\Bigg\|\left(\prod_{t=k}^{k+l-1}A(t)\right)P(k)x\Bigg\| &> L_d\beta_d^l\Big\|P(k)x\Big\| \\
\Bigg\|\left(\prod_{t=k}^{k+l-1}A(t)\right)\left(I-P(k)\right)x\Bigg\| &< \frac{1}{L_d\beta_d^l}\Big\|\left(I-P(k)\right)x\Big\|
\end{align*}
for any $x \in \mathbb{R}^N$.
\end{definition}

Under the assumption of exponential dichotomy (Definition \ref{exp_dichotomy}), it can be proven (\cite {minimum-entropy, exponential_dich}) there exists a bounded sequence of matrices $\{T(t)\}$ with bounded inverses, such that the system matrices pair, $(A(t), B(t))$, can be transformed into stable and antistable components, i.e.,
\begin{align}
\left[
\begin{array}{c|c}
T(t+1)A(t)T(t)^{-1} & T(t+1)B(t)
\end{array}
\right] = \left[
\begin{array}{cc|c}
A_u(t) & 0 & B_u(t)\\
0 & A_s(t) & B_s(t)
\end{array}
\right],
\end{align}
where $A_s(t)$ is exponentially stable and $A_u(t)$ is exponentially antistable (Definition \ref{exp_stab_unstab}). We now make the following assumption on system dynamics.
\begin{assumption}[Stable and antistable] \label{assumption_split} We assume the system matrices, $\{A(t)\}$, possesses an exponential dichotomy. Hence, there exists a change of coordinates, $\{T(t)\}$, such that the system may be transformed into a block diagonal form with stable and antistable components. Henceforth, with no loss of generality, we assume that system pair, $(A(t),B(t))$, is already decomposed into exponentially stable and exponentially antistable components i.e.,
\begin{align}
\left[
\begin{array}{c|c}
A(t) & B(t)
\end{array}
\right] = \left[
\begin{array}{cc|c}
A_u(t) & 0 & B_u(t)\\
0 & A_s(t) & B_s(t)
\end{array}
\right].
\end{align}
\end{assumption}

Our objective is to design a linear state feedback controller, $u(t)=K(t)x(t)$, so that the feedback control system (\ref{stoc_ctrl_eqn}) is mean square exponentially stable (Definition \ref{EMSS}).
\begin{align}
\label{stoc_ctrl_eqn}
x(t+1) &= \left(A(t) + \mu B(t)K(t) + \Delta(t)B(t)K(t)\right)x(t):= \mathcal{A}(t,\Delta(t))x(t).
\end{align}
\begin{definition}[Mean Square Exponential Stability]\label{EMSS}
The system (\ref{stoc_ctrl_eqn}) is said to be mean square exponentially stable, if there exists positive constants $K < \infty$ and $\beta<1$, such that
\begin{align*}
E_{\Delta_0^t}\left[\parallel x(t+1)\parallel^2\right] \leq K\beta^t \parallel x(0)\parallel^2,
\end{align*}
for all $x(0) \in \mathbb{R}^N$, where $E_{\Delta_0^t}[\cdot]$ is the expectation over the sequence $\{\Delta(k)\}_{k=0}^t$.
\end{definition}

It is well known, that the stability information for an LTV system cannot be obtained from the eigenvalues of the time varying matrix computed at each fixed time, $t$ \cite{Khalil_book}. However, stability information for the LTV system can be obtained using {\it Lyapunov exponents}. The Multiplicative Ergodic Theorem (MET) provides technical conditions for the existence of Lyapunov exponents (\cite{Ruelle_ergodic} Proposition 1.3). Before we proceed with the definition of Lyapunov exponents, we provide a definition for the exterior powers of the matrices \cite{arnold_book_rds} (Chap. 3, Lemma 3.2.6).
\begin{definition}
Let '$\wedge$' denote the usual outer product between two quantities and $V$ be a real $n$-dimensional vector space. Let $\wedge^q V$ denote alternating $k$-linear forms. Suppose $M\colon V\to V$ is a linear operator. Then, for $u_1,\ldots,u_q \in V$, the linear extension of
\begin{align*}
M^{\wedge q}(u_1 \wedge \cdots \wedge u_q) = Mu_1 \wedge \cdots \wedge Mu_q
\end{align*}
defines a linear operator, $M^{\wedge q}\colon \wedge^q V \to \wedge^q V$.
\end{definition}

\begin{definition}[Lyapunov exponents \cite{Ruelle_ergodic} Proposition 1.3]\label{Lyapunov_exponents}
Let $\{L(t)\}_{t>0}$ be a sequence of real $m\times m$, matrices such that
\begin{equation*}
\lim_{t\to \infty}\sup\frac{1}{t}\log\parallel L(t)\parallel \leq 0.
\end{equation*}

Define $\mathcal{L}(t) := L(t) L(t-1) \dots L(1)$. Furthermore, suppose the following limits exist
\begin{align}
\lim_{t\to \infty}\frac{1}{t}\log\parallel \mathcal{L}(t)^{\wedge q} \parallel.
\end{align}

Then, the limit
\begin{align}
\Lambda = \lim_{t\to \infty}(\mathcal{L}(t)'\mathcal{L}(t))^{\frac{1}{2t}}
\end{align}
exists. Let $\lambda^i_{exp}$ for $i=1,\ldots,N$ be the eigenvalues of $\Lambda$, such that $\lambda^1_{exp} \geq \lambda^2_{exp} \geq \cdots \geq \lambda^N_{exp}$. Then, the Lyapunov exponents, $\Lambda_{exp}^i$ for $i=1,\ldots, m$, for the system $x(t+1)=L(t)x(t)$ are defined as $\Lambda^i_{exp}=\log \lambda^i_{exp}$. Furthermore, if $\det\left(\Lambda\right)\neq 0$, then
\begin{align}
\label{det_sum_exp}
\lim_{t\to \infty}\frac{1}{t}\log\left|\det\left(\mathcal{L}(t)\right)\right| =\log \prod_{k=1}^N \lambda^k_{exp}(x).
\end{align}
\end{definition}

\begin{remark} The Lyapunov exponents can be used for the stability analysis of the LTV system. In particular, if the maximum Lyapunov exponent of the system $x(t+1)=L(t)x(t)$ is negative, i.e., $\Lambda_{exp}^1<0$, then the system is exponentially stable \cite{arnold_book_rds}.
\end{remark}

\begin{assumption} We assume the Lyapunov exponents for the uncontrolled system $x(t+1)=A(t)x(t)$ are well defined, and there are $0<N_1\leq N$ positive Lyapunov exponents, and $N_2:=N-N_1$ negative Lyapunov exponents.
\end{assumption}

\section{Main Results}
\label{ctrlr}
In this section, we prove the main results of this paper for the limitation on control over uncertain channels in actuation. We use mean square exponential stability of the closed-loop systems as the stability metric. Our first theorem provides a Lyapunov function-based necessary condition for the mean square exponential stability of uncertain feedback control system (\ref{stoc_ctrl_eqn}).

\begin{theorem}
\label{lyap_thm}
The feedback control system (\ref{stoc_ctrl_eqn}) is mean square exponentially stable only if there exists a sequence of positive definite matrices $\{P(t)\}_{t\geq 0}$ and positive constants $\alpha_1$ and $\alpha_2$, such that the following conditions are satisfied.
\begin{align}
\label{MSS_lyap_fn_condn}
E_{\Delta(t)}\left[\mathcal{A}'(t,\Delta(t))P(t+1)\mathcal{A}(t,\Delta(t)) \right] < P(t),\;\;\; \alpha_1 I < P(t) < \alpha_2 I,
\end{align}
for all $t\geq 0$, and ${\cal A}(t,\Delta(t)):=A(t)+(\mu +\Delta)B(t)K(t)$ from Eq. (\ref{stoc_ctrl_eqn}).
\end{theorem}
\begin{proof}
Consider the following construction of $P(t)$,
\begin{align*}
P(t) = \sum_{n=t}^{\infty}E_{\Delta_t^n}\left[ \left(\prod_{k=t}^n\mathcal{A}(k,\Delta(k))\right)' \left(\prod_{k=t}^n\mathcal{A}(k,\Delta(k))\right) \right],
\end{align*}
where $E_{\Delta_1^n}[\cdot]$ means expectation has been taken over $\Delta(t)$ for $t = 1,\ldots,n$. Since the closed-loop system, $x(t+1)={\cal A}(t,\Delta(t))$, is assumed mean square exponentially stable, the construction for $P(t)$ is well defined.
We can also write the above equation as
\begin{align*}
E_{\Delta(t)}\left[\mathcal{A}'(t,\Delta(t))\mathcal{A}(t,\Delta(t)) + \mathcal{A}'(t,\Delta(t))P(t+1)\mathcal{A}(t,\Delta(t))\right] = P(t).
\end{align*}
The equation for $P(t)$ can be rewritten as follows:
\begin{align*}
E_{\Delta(t)}\left[\mathcal{A}'(t,\Delta(t))P(t+1)\mathcal{A}(t,\Delta(t))\right] - P(t) = -E_{\Delta(t)}\left[\mathcal{A}'(t,\Delta(t))\mathcal{A}(t,\Delta(t))\right].
\end{align*}

Since ${\cal A}(t,\Delta(t))$ is invertible for $\Delta(t)=0$ and is continuous with respect to $\Delta$, it follows that $E_{\Delta(t)}\left[\mathcal{A}'(t,\Delta(t))\mathcal{A}(t,\Delta(t))\right] > 0$. Hence, we obtain
\begin{align*}
E_{\Delta(t)}\left[\mathcal{A}'(t,\Delta(t))P(t+1)\mathcal{A}(t,\Delta(t))\right] < P(t).
\end{align*}
We now need to show $P(t)$ is bounded. The system is assumed mean square exponentially stable as given in Definition \ref{EMSS}. There exists $\beta < 1$ and $K < \infty$, such that
\begin{align*}
E_{\Delta_0^t}\left[\parallel x(t+1)\parallel^2\right] =
E_{\Delta_0^t}\left[ \bigg\| \prod_{k=0}^t \mathcal{A}(k,\Delta(k))x(0) \bigg\|^2 \right] \leq K\beta^t\parallel x(0) \parallel^2.
\end{align*}
Hence, we have
$\parallel P(t) \parallel \leq K \sum_{k=0}^{\infty}\beta^k = \frac{K}{1 - \beta}$. Since matrix $\mathcal{A}(t,\Delta(t))$ is bounded below in some Lebesgue neighborhood of $\Delta(t)=0$ for all $t\geq 0$, we have some constant $\alpha_1 > 0$, such that $\alpha_1 I \leq E_{\Delta(t)}\left[\mathcal{A}'(t,\Delta(t))\mathcal{A}(t,\Delta(t))\right] \; \forall t \geq 0$ which gives
\begin{align*}
P(t) &\geq E_{\Delta(t)}[\mathcal{A}'(t,\Delta(t))\mathcal{A}(t,\Delta(t))] \geq \alpha_1 I.
\end{align*}
Setting $\alpha_2 = \frac{K}{1 - \beta}$, we get $\alpha_1 I \leq P(t) \leq \alpha_2 I$.
\end{proof}

We have the following Lemma providing necessary conditions for the mean square exponential stability of (\ref{stoc_ctrl_eqn}) in terms of the solution of the Riccati equation.

\begin{lemma}
\label{ctrl_gain_thm}
The necessary condition for mean square exponentially stability of system (\ref{stoc_ctrl_eqn}) derived in Theorem \ref{lyap_thm} is equivalent to
\begin{align}
P_0(t) > A'(t)P_0(t+1)A(t)-\frac{\mu^2}{\mu^2+\sigma^2}A'(t)P_0(t+1)B(t)\left(B'(t)P_0(t+1)B(t)\right)^{-1}B'(t)P_0(t+1)A(t),
\end{align}
where $\mu$ is the mean connectivity, $\sigma^2$ is the variance of the zero mean uncertainty. $\{P_0(t)\}_{t\geq 0}$ is the sequence of positive definite symmetric matrices that satisfies the
following Riccati equation \cite{kwakernaak}.
\begin{align*}
P_0(t) = A'(t)P_0(t+1)A(t)-A'(t)P_0(t+1)B(t)\left(I_M + B'(t)P_0(t+1)B(t)\right)^{-1}B'(t)P_0(t+1)A(t) + R(t),
\end{align*}
where $R(t) = R'(t) > 0$ is such that $P_0(t)$ is uniformly bounded above and below.
\end{lemma}

\begin{proof}
From Theorem \ref{lyap_thm} we know a necessary condition for mean square exponential stability of (\ref{stoc_ctrl_eqn}) is given by
\begin{align*}
P(t) > E_{\Delta(t)}\left[\mathcal{A}'(t,\Delta(t))P(t+1)\mathcal{A}(t,\Delta(t))\right],
\end{align*}
where $\mathcal{A}(t,\Delta(t)) := A(t) + \mu B(t)K(t) + \Delta(t)B(t)K(t)$ and there exist $\alpha_1,\alpha_2 > 0$, such that $\alpha_1 I < P(t) < \alpha_2 I$ for all $t >0$. Expanding the above equation, we derive the necessary condition
\begin{align}
\label{nec_suff_condn}
P(t) &> A'(t)P(t+1)A(t) - \mu A'(t)P(t+1)B(t)K(t)\nonumber\\
&\quad \quad- \mu K'(t)B'(t)P(t+1)A(t) + \left(\mu^2 + \sigma^2\right)K'(t)B'(t)P(t+1)B(t)K(t).
\end{align}
Taking the trace and minimizing the RHS w.r.t. $K(t)$, we obtain optimal $K^*(t)$ \cite{kwakernaak} to achieve the mean square exponential stability as
\begin{align*}
K^*(t) = -\frac{\mu}{\mu^2+\sigma^2}\left(B'(t)P(t+1)B(t)\right)^{-1}B'(t)P(t+1)A(t).
\end{align*}
This provides us the necessary condition for mean square exponential stability of the controlled system (\ref{stoc_ctrl_eqn})
\begin{align}
\label{riccati_necessary}
P(t) &> A'(t)P(t+1)A(t)-\frac{\mu^2}{\mu^2+\sigma^2}A'(t)P(t+1)B(t)\left(B'(t)P(t+1)B(t)\right)^{-1}B'(t)P(t+1)A(t).
\end{align}
Now, using the fact $P(t)$ is bounded below, there exists $\Sigma > 0$, such that $\frac{\sigma^2}{\mu^2}B'(t)P(t+1)B(t) \geq \Sigma I_M$ for all $t \geq 0$. Substituting this in (\ref{riccati_necessary}) we obtain,
\begin{align}
\label{riccati_1}
P(t) &> A'(t)P(t+1)A(t)-A'(t)P(t+1)B(t)\left(\Sigma I_M + B'(t)P(t+1)B(t)\right)^{-1}B'(t)P(t+1)A(t).
\end{align}
Defining $P_0(t) = \frac{1}{\Sigma}P(t)$ we find
\begin{align}
\label{riccati_2}
P_0(t) &> A'(t)P_0(t+1)A(t)-A'(t)P_0(t+1)B(t)\left(I_M + B'(t)P_0(t+1)B(t)\right)^{-1}B'(t)P_0(t+1)A(t).
\end{align}
Thus, there exists $R(t) > 0$ for all $t\geq 0$ as given in \cite{kwakernaak_book}, such that
\begin{align}
\label{riccati_2_ineq}
P_0(t) &= A'(t)P_0(t+1)A(t)-A'(t)P_0(t+1)B(t)\left( I_M + B'(t)P_0(t+1)B(t)\right)^{-1}B'(t)P_0(t+1)A(t) + R(t).
\end{align}
We notice that (\ref{riccati_necessary}) is independent of any constant scaling. Hence, $P_0(t)$ satisfies
\begin{align}
P_0(t) &> A'(t)P_0(t+1)A(t)-\frac{\mu^2}{\mu^2+\sigma^2}A'(t)P_0(t+1)B(t)\left(B'(t)P_0(t+1)B(t)\right)^{-1}B'(t)P_0(t+1)A(t).
\end{align}
This gives the required necessary condition.
\end{proof}
The first main result of the paper provides a computable necessary condition for stability of (\ref{stoc_ctrl_eqn}) for $M$-input case with $M < N$.

\begin{theorem}
\label{single_input_thm}
A necessary condition for the mean square exponential stability of system (\ref{stoc_ctrl_eqn}) for $M<N$ inputs is given by
\begin{align}
\label{MSS_exponent}
\sigma^2\frac{\left(\prod_{k=1}^{N_1} \lambda_{exp}^k\right)^{\frac{2}{M}} - 1}{\mu^2} < 1,
\end{align}
where $\sigma^2 < \infty$ is the variance of uncertainty $\Delta$ ( Eq. (\ref{delta_uncertainty})), and $\lambda_{exp}^k = e^{\Lambda_{exp}^k}$, and $\Lambda_{exp}^k$ is the $k^{th}$ positive Lyapunov exponent of uncontrolled system
$ x(t+1) =  A(t) x(t)$ for $k=1,\ldots,N_1$.
\end{theorem}

\begin{proof}
From Theorem \ref{ctrl_gain_thm}, we have a necessary condition for a system with $N$ states and $M$ inputs given by
\begin{align}
\label{nec_condn_1}
P_0(t) &> {A}'(t)P_0(t+1){A}(t)-\frac{\mu^2}{\mu^2+\sigma^2}{A}'(t)P_0(t+1){B}(t)\left({B}'(t)P_0(t+1){B}(t)\right)^{-1}{B}'(t)P_0(t+1){A}(t).
\end{align}
Let $P_0(t)$ be given by the blockwise representation,
\begin{align}
\label{P0_block}
P_0(t) := \left[
\begin{array}{cc}
P_{0_{11}}(t) & P_{0_{12}}(t)\\
P_{0_{12}}'(t) & P_{0_{22}}(t)
\end{array}
\right],
\end{align}
where $P_{0_{11}}$ is an $N_1 \times N_1$ block, $P_{0_{22}}$ is an $N_2\times N_2$ block, and $P_{0_{12}}$ is a $N_1\times N_2$ block. We know since the matrix $P_0(t)$ is positive definite, any $k\times k$ block for $k\leq N$ must be positive definite. Hence, from (\ref{nec_condn_1}) and (\ref{P0_block}), the necessary condition for mean square exponential stability provides the positive definiteness of the first $N_1\times N_1$ block in (\ref{nec_condn_1}), given by
\begin{align}
P_{0_{11}}(t) &> A_u'(t)P_{0_{11}}(t+1)A_u(t)\nonumber\\
&-\frac{\mu^2}{\mu^2+\sigma^2}\Big[A_u'(t)\Big(P_{0_{11}}(t+1)B_u(t)+P_{0_{12}}(t+1)B_s(t)\Big)\Big]\nonumber\\
&\times\Big[\left(B'(t)P_0(t+1)B(t)\right)^{-1}\Big(P_{0_{11}}(t+1)B_u(t)+P_{0_{12}}(t+1)B_s(t)\Big)'A_u(t)\Big].
\end{align}
Taking determinants on both sides and using Sylvester's determinant theorem, we obtain
\begin{align}
\label{nec_condn_2}
\det \left(P_{0_{11}}(t)\right) &> \Big[\det \left(A_u(t)\right)^2\det \left(P_{0_{11}}(t+1)\right)\Big] \nonumber\\
&\times\Bigg[\det\bigg(I_{M}-\left(\frac{\mu^2}{\mu^2+\sigma^2}\right)\left({B}'(t)P_0(t+1){B}(t)\right)^{-1}\nonumber\\
&\times\Big(P_{0_{11}}(t+1)B_u(t)+P_{0_{12}}(t+1)B_s(t)\Big)'P_{0_{11}}(t+1)^{-1}\Big(P_{0_{11}}(t+1)B_u(t)+P_{0_{12}}(t+1)B_s(t)\Big)\bigg)\Bigg].
\end{align}
Using the partition for ${B}(t)$, we write
\begin{align}
\label{P_0_compare}
&{B}'(t)P_0(t+1){B}(t)\nonumber\\
&= B_u'(t)P_{0_{11}}(t+1)B_u(t)+B_u'(t)P_{0_{12}}(t+1)B_s(t)+B_s'(t)P_{0_{12}}'(t+1)B_u(t)+B_s'(t)P_{0_{22}}(t+1)B_s(t)\nonumber\\
&> B_u'(t)P_{0_{11}}(t+1)B_u(t)+B_u'(t)P_{0_{12}}(t+1)B_s(t)+B_s'(t)P_{0_{12}}'(t+1)B_u(t)\nonumber\\
&\quad +B_s'(t)\left(P_{0_{12}}(t+1)'P_{0_{11}}(t+1)^{-1}P_{0_{12}}(t+1)\right)B_s(t)\nonumber\\
&>\Big(P_{0_{11}}(t+1)B_u(t)+P_{0_{12}}(t+1)B_s(t)\Big)'P_{0_{11}}(t+1)^{-1}\Big(P_{0_{11}}(t+1)B_u(t)+P_{0_{12}}(t+1)B_s(t)\Big),
\end{align}
since $P_{0_{22}}(t+1) > P_{0_{12}}(t+1)'P_{0_{11}}(t+1)^{-1}P_{0_{12}}(t+1)$ as $P_0(t+1)$ is positive definite. Hence, from (\ref{P_0_compare}) we derive
\begin{align}
\label{P_0_ineq}
I_M &> \Bigg[\left({B}'(t)P_0(t+1){B}(t)\right)^{-1}\Big(P_{0_{11}}(t+1)B_u(t)+P_{0_{12}}(t+1)B_s(t)\Big)'P_{0_{11}}(t+1)^{-1}\nonumber\\
&\quad \times\Big(P_{0_{11}}(t+1)B_u(t)+P_{0_{12}}(t+1)B_s(t)\Big)\Bigg].
\end{align}
Hence, using (\ref{P_0_ineq}) in (\ref{nec_condn_2}), we find
\begin{align}
1 > \left(\frac{\sigma^2}{\mu^2+\sigma^2}\right)^M\det \left(A_u(t)\right)^2\det \left(P_{0_{11}}(t+1)\right)\det \left(P_{0_{11}}(t)\right)^{-1}.
\end{align}
Hence, the necessary condition can be written as
\begin{align*}
1 > \left(\frac{\sigma^2}{\mu^2+\sigma^2}\right)^{M(t+1)}\prod_{k=0}^t\left(\det \left(A_u(k)\right)\right)^2\det \left(P_{0_{11}}(t+1)\right)\det \left(P_{0_{11}}(0)\right)^{-1}.
\end{align*}
Taking the logarithm, averaging over $t+1$, and taking the limit as $t \to \infty$, we obtain the necessary condition,
\begin{align*}
1 > \left(\frac{\sigma^2}{\mu^2+\sigma^2}\right)^M\left(\prod_{k=1}^{N_1} \lambda_{exp}^k\right)^2,
\end{align*}
where we are use the fact $P_0(t)$ (and hence $P_{0_{11}}(t)$) is bounded above and below for all $t\geq 0$ and Eq. (\ref{det_sum_exp}) from Definition \ref{Lyapunov_exponents}. This condition is rewritten as
\begin{align}
\sigma^2\frac{\left(\prod_{k=1}^{N_1} \lambda_{exp}^k\right)^{\frac{2}{M}} - 1}{\mu^2}<1.\label{mss}
\end{align}
\end{proof}

\begin{remark} \label{remark_robust}  The necessary conditions derived in Theorem \ref{lyap_thm}, Lemma \ref{ctrl_gain_thm} are equivalent. Lemma \ref{ctrl_gain_thm} implies Theorem \ref{single_input_thm} though the converse may not be true. Hence Theorem \ref{single_input_thm} is lower in the hierarchy in  comparison with Theorem \ref{lyap_thm} and Lemma \ref{ctrl_gain_thm}.
The necessary condition for stability in Eq. (\ref{MSS_exponent}) can be used to provide critical value of variance, $\sigma^{*}$, above which the system is guaranteed to be mean square unstable. In particular the critical value of variance using Eq. (\ref{MSS_exponent}) is given by
\begin{eqnarray}
\sigma^{*}=\left(\frac{\mu^{2}}{(\prod_{k=1}^{N_1} \lambda_{exp}^k)^{\frac{2}{M}} -1}\right)^{\frac{1}{2}}\label{critical_prob}.
\end{eqnarray}
The necessary condition for mean square exponential stability derived in the above theorem  is tighter for the single input case (i.e., $M=1$).
However, for $1<M<N$, Eq. (\ref{MSS_exponent}) provides a necessary condition for stability and can be made tighter, i.e.,  improved necessary condition can be obtained that will provide for a smaller value of critical variance $\sigma^{*}$ than the one provided by Eq. (\ref{critical_prob}). We expect the tighter necessary condition to depend on some combination of Lyapunov exponents and not necessarily on all the Lyapunov exponents as it does in Eq. (\ref{MSS_exponent}). Borrowing terminology from \cite{Elia2005237}, the quantity $\left(\frac{\left(\prod_{k=1}^{N_1} \lambda_{exp}^k\right)^{\frac{2}{M}} - 1}{\mu^2}\right)^{\frac{1}{2}}$ in Eq. (\ref{MSS_exponent}) can be viewed as the scaled mean square norm of the nominal system with mean connectivity $\mu$ as seen by the uncertainty $\Delta$ for block diagram in Fig. \ref{fig:sch}b. Thus if we consider this as the mean square input-output gain of the nominal system and $\sigma^2$ as the mean square gain of the uncertainty, then the necessary condition in Theorem \ref{single_input_thm} may be interpreted as a necessary small gain condition for mean square stability of nonlinear systems.
\end{remark}

The next main result of this paper provides the necessary and sufficient condition for the mean square exponential stability of the feedback system for the $N$ input case. For this $N$ input case, we assume the matrix $B(t)$ is non-singular.

\begin{theorem}
\label{n_input_thm}
A necessary and sufficient condition for the mean square exponential stability of (\ref{stoc_ctrl_eqn}) with $N$ inputs is given by
\begin{align}
\sigma^2\frac{\left(\lambda_{exp}^1\right)^2 - 1}{\mu^2} < 1,
\end{align}
where $\lambda_{exp}^1 = e^{\Lambda_{exp}^1}$ and $\Lambda_{exp}^1$ is the maximum positive Lyapunov exponent of system $x(t+1) = A(t)x(t)$.
\end{theorem}

\begin{proof}
From Theorem \ref{ctrl_gain_thm}, we obtain the following necessary condition for mean square exponential stability
\begin{align*}
P_0(t) &> A'(t)P_0(t+1)A(t)- \frac{\mu^2}{\mu^2+\sigma^2}A'(t)P_0(t+1)B(t)\left(B'(t)P_0(t+1)B(t)\right)^{-1}B'(t)P_0(t+1)A(t).
\end{align*}
Since $B(t)$ is a non-singular $N\times N$ matrix and $P_0(t)$ is invertible, we can write the above Lyapunov function inequality as
\begin{align}
\label{N_lyap}
P_0(t) > \frac{\sigma^2}{\mu^2+\sigma^2}A'(t)P_0(t+1)A(t).
\end{align}
%and the controller gain is given by $K = - \frac{\mu}{\mu^2 + \sigma^2}B(t)^{-1}A(t)$.
Equation (\ref{N_lyap}) implies following inequality to be true
\begin{align*}
P_0(0) > \left(\frac{\sigma^2}{\mu^2+\sigma^2}\right)^{t+1}\left(\prod_{k=0}^tA(k)\right)'P_0(t+1)\left(\prod_{k=0}^tA(k)\right).
\end{align*}
Since there exists $\alpha_1 > 0$ and $\alpha_2 > 0$, such that $\alpha_2 I > P_0(t) > \alpha_1 I$ for all $t > 0$, the necessary condition can be written as
\begin{align}
\label{nec_N_ipt}
\frac{\alpha_2}{\alpha_1} I > \left(\frac{\sigma^2}{\mu^2+\sigma^2}\right)^{t+1}\left(\prod_{k=0}^tA(k)\right)'\left(\prod_{k=0}^tA(k)\right).
\end{align}
Take the logarithm in (\ref{nec_N_ipt}), divide by $t+1$, and take $\lim_{t\to \infty}$, we get the following necessary condition for mean square exponentially stability,
\begin{align}
\label{nec_N_ipt_LE}
\frac{\sigma^2}{\mu^2+\sigma^2}\Lambda^2 < 1,
\end{align}
which is satisfied only if $\frac{\sigma^2}{\mu^2+\sigma^2}\left(\lambda_{exp}^1\right)^2 < 1$. This can be rewritten as $\sigma^2\frac{\left(\lambda_{exp}^1\right)^2 - 1}{\mu^2} < 1$.

We will now prove the sufficiency part.
Consider the controller gain as derived in the necessary condition given by $K = -\frac{\mu}{\mu^2+\sigma^2}B(t)^{-1}A(t)$. Using this controller gain, the dynamics of the controlled system are given by
\begin{align}
\label{dyn_N_ipt}
x(t+1) = \frac{\sigma^2 - \Delta(t)\mu}{\mu^2 + \sigma^2}A(t)x(t).
\end{align}
From (\ref{dyn_N_ipt}), we obtain
\begin{align}
\label{exp_val_Nipt}
E_{\Delta(t)}\left[||x(t+1)||^2\right] = \frac{\sigma^2}{\sigma^2 + \mu^2}x'(t)A'(t)A(t)x(t).
\end{align}
Thus, from (\ref{exp_val_Nipt}) we obtain
\begin{align}
\label{exp_val_Nipt1}
E_{\Delta_0^t}\left[||x(t+1)||^2\right] = \left(\frac{\sigma^2}{\sigma^2 + \mu^2}\right)^{t+1}x'(0)\left(\prod_{k=0}^tA(k)\right)'\left(\prod_{k=0}^tA(k)\right)x(0).
\end{align}
Now, we claim there exist positive constants $K < \infty$ and $\beta < 1$, such that
\begin{align}
\label{full_seq_ineq}
\left(\frac{\sigma^2}{\sigma^2+\mu^2}\right)^{t+1}\Bigg\|\left(\prod_{k=0}^tA(k)\right)'\left(\prod_{k=0}^tA(k)\right)\Bigg\|_2 < K\beta^{t+1}
\end{align}
for all $t \geq 0$. We will defer the proof of this claim for later to maintain continuity in the proof of the sufficiency condition. Now, using the claim from (\ref{full_seq_ineq}) in (\ref{exp_val_Nipt}), we derive
\begin{align}
E_{\Delta_0^t}\left[||x(t+1)||^2\right] &= \left(\frac{\sigma^2}{\sigma^2 + \mu^2}\right)^{t+1}x'(0)\left(\prod_{k=0}^tA(k)\right)'\left(\prod_{k=0}^tA(k)\right)x(0)\nonumber\\
&\leq \left(\frac{\sigma^2}{\sigma^2+\mu^2}\right)^{t+1}\Bigg\|\left(\prod_{k=0}^tA(k)\right)'\left(\prod_{k=0}^tA(k)\right)\Bigg\|_2 ||x(0)||^2 < K\beta^{t+1}||x(0)||^2.
\end{align}
Thus, we have proven the required sufficiency condition.
We will now prove the claim made in (\ref{full_seq_ineq}). To prove this claim suppose
\begin{align}
\sigma^2\frac{\left(\lambda_{exp}^1\right)^2 - 1}{\mu^2} < 1.
\end{align}
Hence, there exists $\beta < 1$, such that $\frac{\sigma^2}{\sigma^2 + \mu^2}\left(\lambda_{exp}^1\right)^2 = \beta^2$. Furthermore, from the definition of the Lyapunov exponents (Definition \ref{Lyapunov_exponents}), we obtain $\lambda_{exp}^1$ for the system $x(t+1) = A(t)x(t)$ is given by
\begin{align}
\lambda_{exp}^1 = ||\Lambda||_2,
\end{align}
where $||\Lambda||_2$ is the matrix $2$-norm of the matrix $\Lambda$ given by
\begin{align}
\Lambda := \lim_{t\to \infty}\left(\prod_{k=0}^tA(k)\right)'\left(\prod_{k=0}^tA(k)\right)^{\frac{1}{2(t+1)}}.
\end{align}
Thus, from the property of the matrix $2$-norm and the Lyapunov exponent definition (Proposition 1.3 \cite{Ruelle_ergodic}, \cite{Ruelle85}), we have
\begin{align}
\lambda_{exp}^1 = \lim_{t\to \infty}\Bigg\|\left(\prod_{k=0}^tA(k)\right)'\left(\prod_{k=0}^tA(k)\right)\Bigg\|_2^{\frac{1}{2(t+1)}}.
\end{align}
Thus, we can conclude
\begin{align}
\label{seq_limit}
\lim_{t\to\infty}\left(\left(\frac{\sigma^2}{\sigma^2 +\mu^2}\right)^{t+1}\Bigg\|\left(\prod_{k=0}^tA(k)\right)'\left(\prod_{k=0}^tA(k)\right)\Bigg\|_2\right)^{\frac{1}{t+1}}
&=\frac{\sigma^2}{\sigma^2+\mu^2}\lim_{t\to\infty}\Bigg\|\left(\prod_{k=0}^tA(k)\right)'\left(\prod_{k=0}^tA(k)\right)\Bigg\|_2^{\frac{1}{t+1}}\nonumber\\
&= \frac{\sigma^2}{\sigma^2 + \mu^2}\left(\lambda_{exp}^1\right)^2\nonumber\\
&= \beta^2 < \beta < 1.
\end{align}
Hence, there exists $N_{\beta} < \infty$, such that $\left(\frac{\sigma^2}{\sigma^2+\mu^2}\right)^{t+1}\Big\|\Big(\prod_{k=0}^tA(k)\Big)'\Big(\prod_{k=0}^tA(k)\Big)\Big\|_2 < \beta^{t+1}$ for all $t \geq N_{\beta}$. The exixstence of $N_{\beta} < \infty$ is proved by contradiction as follows. Suppose $N_{\beta} < \infty$ does not exist. Thus there exits a subsequence $\{t_i\}_{i \geq 0}$, such that $\left(\frac{\sigma^2}{\sigma^2+\mu^2}\right)^{t_i+1}\Big\|\Big(\prod_{k=0}^{t_i}A(k)\Big)'\Big(\prod_{k=0}^{t_i}A(k)\Big)\Big\|_2 > \beta^{t_i+1}$. 

Now, let $s(t) := \left(\left(\frac{\sigma^2}{\sigma^2+\mu^2}\right)^{t+1}\Big\|\Big(\prod_{k=0}^{t}A(k)\Big)'\Big(\prod_{k=0}^{t}A(k)\Big)\Big\|_2\right)^{\frac{1}{t+1}}$. We have $s(t_i) > \beta$ for all $i \geq 0$. Hence, we have from \cite{Rudin_principles} and (\ref{seq_limit})
\begin{align}
\beta \leq \limsup_{i\to \infty}s(t_i) \leq \limsup_{t\to \infty}s(t) = \beta^2 < \beta,
\end{align}
a contradiction. Thus, we conclude there exists an $N_{\beta} < \infty$, such that
\begin{align}
\label{seq_ineq}
\left(\frac{\sigma^2}{\sigma^2+\mu^2}\right)^{t+1}\Bigg\|\left(\prod_{k=0}^tA(k)\right)'\left(\prod_{k=0}^tA(k)\right)\Bigg\|_2 < \beta^{t+1},
\end{align}
for all $t \geq N_{\beta}$. Now, we define
\begin{align}
\label{const_def}
K := \max\Bigg\{1,sup_{0 \leq t \leq N_{\beta}}\frac{\left(\frac{\sigma^2}{\sigma^2+\mu^2}\right)^{t+1}\Big\|\left(\prod_{k=0}^tA(k)\right)'\left(\prod_{k=0}^tA(k)\right)\Big\|_2}{\beta^{t+1}}\Bigg\}
\end{align}
As the supremum is taken over a finite sequence, it will exist and be finite. Hence, from (\ref{seq_ineq}) and (\ref{const_def}), there exist positive constants $K < \infty$ and $\beta < 1$, such that
\begin{align}
\label{full_seq_ineq1}
\left(\frac{\sigma^2}{\sigma^2+\mu^2}\right)^{t+1}\Bigg\|\left(\prod_{k=0}^tA(k)\right)'\left(\prod_{k=0}^tA(k)\right)\Bigg\|_2 < K\beta^{t+1}.
\end{align}
for all $t \geq 0$. This proves the required sufficient condition.
\end{proof}

\begin{remark} \label{remark_NSingle} We examine the two different stability conditions derived in Theorems \ref{single_input_thm} and \ref{n_input_thm} for single input and $N$ input case, respectively. We notice that the necessary condition for the single input case is a function of all positive Lyapunov exponents of the system; whereas, the condition for $N$ input case is a function of only the largest positive Lyapunov exponent. Intuitively, the difference in conditions can be explained as follows. The analysis of an $N$ state system with $N$ inputs is similar to that of $N$ parallel scalar systems with $N$ parallel input channels. Thus, one derives conditions for stabilization for each individual system. The stabilization condition for each system then depends upon the Lyapunov exponent of individual system and the most restrictive of these $N$ conditions provides the stability condition for the entire system. On the other hand, for an $N$-state single input system, the lone input is responsible for stabilizing all the states. The sum of positive Lyapunov exponents (or the product of exponential of the Lyapunov exponent) is equal to the entropy of a system and is a measure of the rate of expansion of the volume in the state space. For stability in a single input case, we require this expansion of open-loop dynamics be compensated by the controller. Hence, the condition for a single input case turns out to be a function of the sum of all positive Lyapunov exponents of the open-loop system.
\end{remark}

\section{Simulations}
\label{sim}

In this section, we present simulation results for the controller design for LTV systems in the presence of the stochastic uncertain channel for a single input system. The uncertain channel considered in the simulations, is an erasure channel modeled as a Bernoulli random variable. Although the main results of this paper provide only necessary conditions for the mean square exponential stability, the simulation results show the derived necessary condition is close to be sufficient.
\subsection{Example 1}
We consider the continuous time LTV system as described in \cite{Khalil_book} by $\dot x(t)=A(t) x(t) + \gamma(t)B u(t)$
with $B = [1\quad 1]'$. The eigenvalues of $A(t)$ are located in the left-half plane at $-0.25\pm j 0.25 \sqrt{7}$ and; hence, independent of $t$. However, the origin is exponentially unstable. This can be verified from the state transition matrix for $A(t)$ written as follows \cite{Khalil_book}:
\begin{align*}
\Phi(t,0) =
\left(
\begin{array}{cc}
e^{0.5t}\cos(t) & e^{-t}\sin(t) \\
-e^{0.5t}\sin(t)  & e^{-t}\cos(t)
\end{array}
\right).
\end{align*}
The state transition matrix can be used to construct a discrete time system as follows:
\begin{align}
x(\Delta (t+1))=\Phi(\Delta(t+1),0)\Phi^{-1}(\Delta t,0)\Phi(\Delta t,0)x(0)
=: \mathcal{A}(\Delta t)x(\Delta t),
\end{align}
where $\mathcal{A}(\Delta t)=\Phi(\Delta (t+1),0)\Phi^{-1}(\Delta t,0)$. For $\Delta =0.1$, the Lyapunov exponents of the system are computed equal to $\lambda_1=0.05$ and $\lambda_2=-0.1$.

\begin{figure}
\centering
\includegraphics[width=0.46\textwidth]{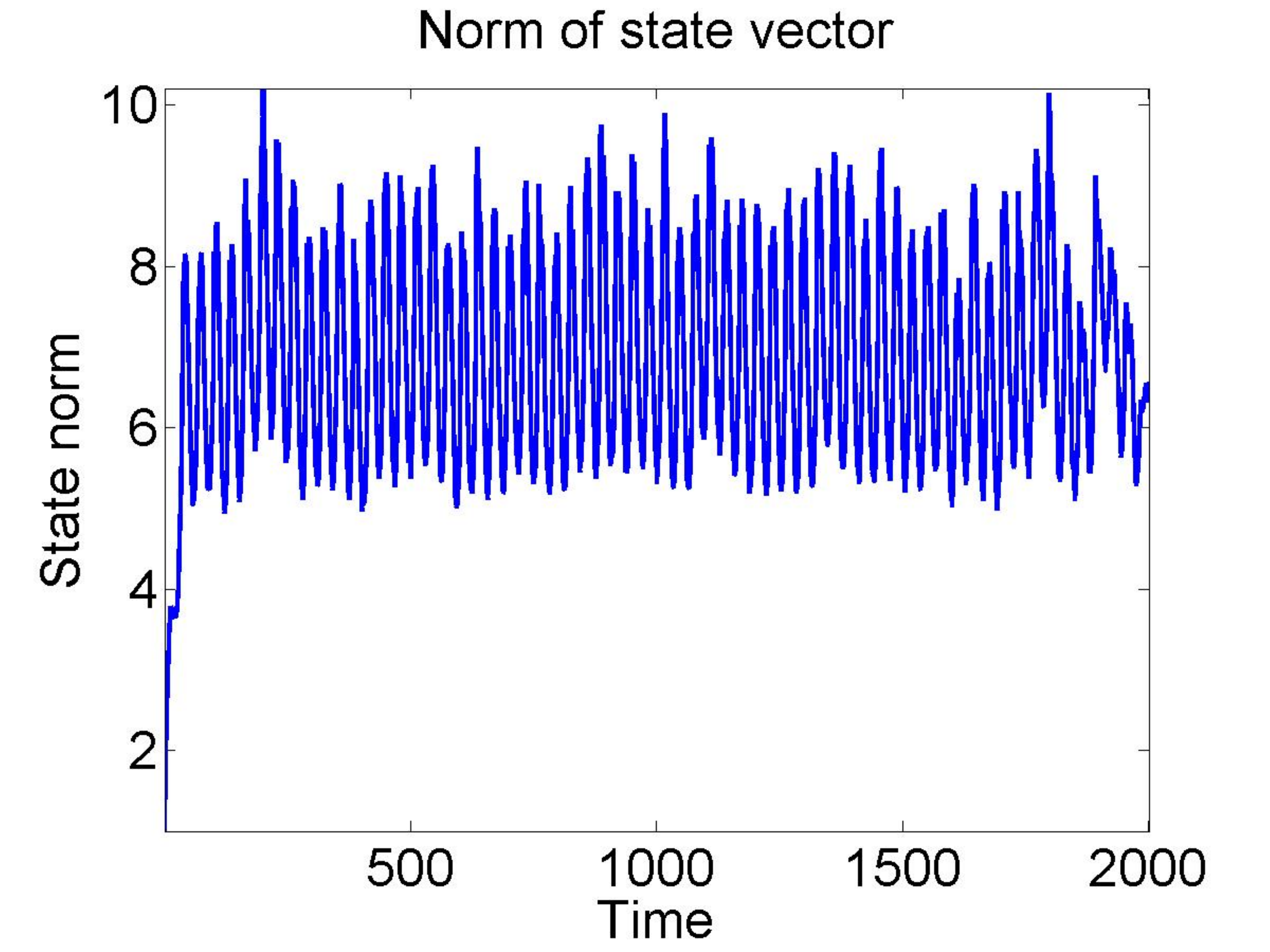}
\caption{State norm for non-erasure probability $p^* < p = 0.11$.}
\label{fig:fig1}
\end{figure}

\begin{figure}
\centering
\includegraphics[width=0.46\textwidth]{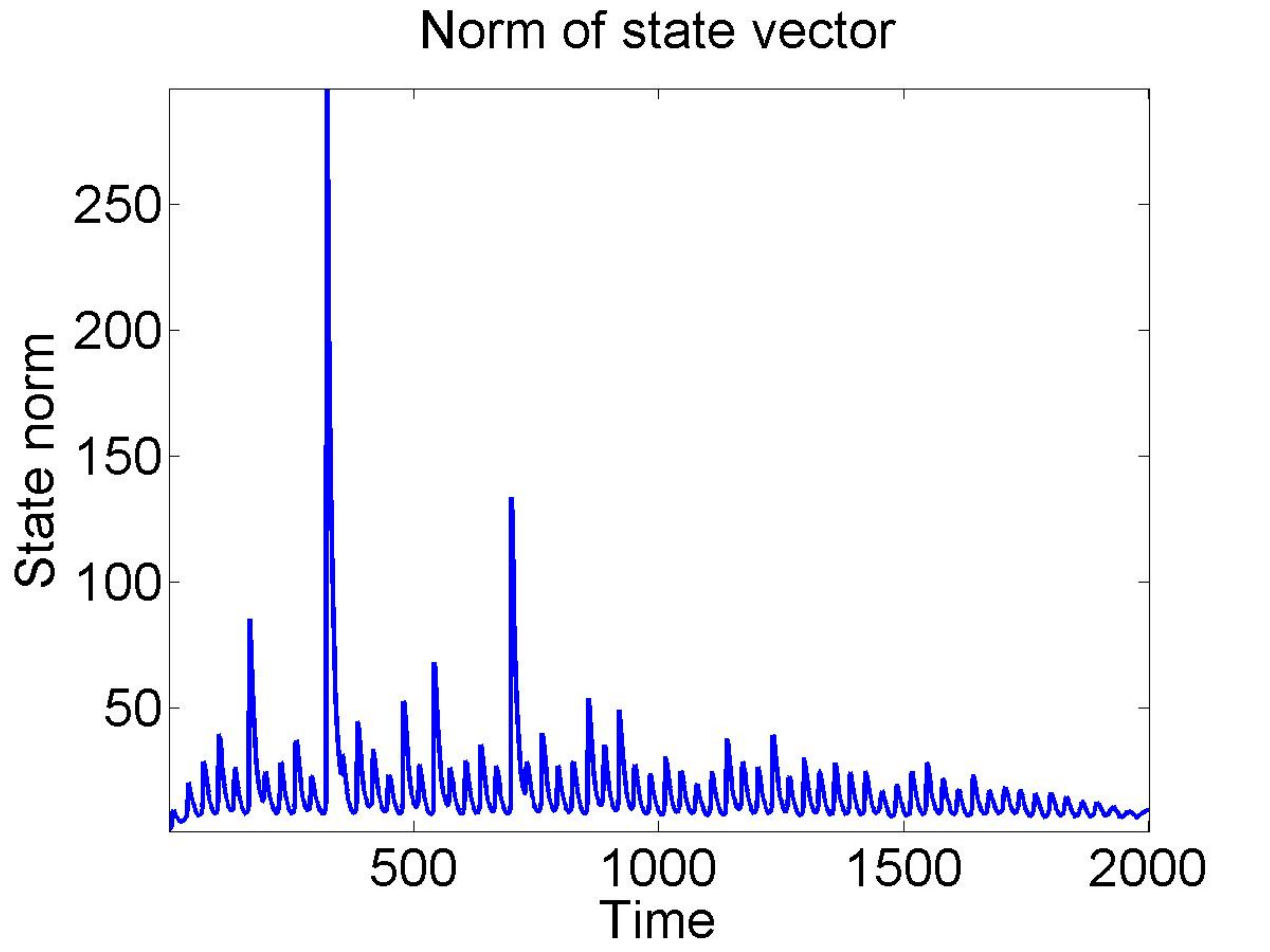}
\caption{State norm for non-erasure probability $p^* > p = 0.09$.}
\label{fig:fig2}
\end{figure}
The critical probability, $p^*$, is the function of the positive Lyapunov exponent and computed equal to $p^*=1-\frac{1}{e^{2\lambda_1}}=0.0952$. In Figs. (\ref{fig:fig1}) and (\ref{fig:fig2}), we show the plots for the state norm for non-erasure probability above and below the critical value of $p^{*}$, respectively. The plots are obtained by averaging the state norm over $1000$ different realizations of the Bernoulli random variable. A zero mean white Gaussian noise with unit variance is added to the system to visualize the mean square unstable dynamics. We see for $p=0.09 < p^*$, the state norm fluctuates to substantially high values, while for $p=0.11 > p^*$, the state norm stabilizes to a small band an order of magnitude smaller than the values at $p=0.09$. The small asymptotic variance for the case of $p = 0.11 > p^*$ is due to the addition of the Gaussian noise vector, and will decrease as the noise variance is decreased. Thus, we may conclude for the controlled system to be robust to the actuation link failure uncertainty in the exponential mean square sense, the probability of non-erasure must be at least given by $p^* = 0.095$. Furthermore, the condition given by the positive Lyapunov exponent seems sufficient, as a small increase in non-erasure probability above $p^*$ shows mean square stable behavior.

\subsection{Example 2}
In the next example, we choose a linear time periodic system with all Lyapunov exponents positive. The system is given by the following sets of periodic $A$ and $B$ matrices
\begin{align*}
A_1 = \left(
\begin{array}{ccc}
-0.4   & 0.8  &  1.2\\
    1  &  0.8  & -0.4\\
    0.6  & -0.8 &   0.4
\end{array}
\right),\quad B_1 = \left(
\begin{array}{c}
1\\
1\\
1
\end{array}
\right)
\end{align*}
\begin{align*}
A_2 = \left(
\begin{array}{ccc}
1.6   & -1.4  &  1.2\\
0.8  &  -1.6  & 2.8\\
 1.6  & -2.2 &   1.2
\end{array}
\right),\quad B_2 = \left(
\begin{array}{c}
2\\
1\\
1
\end{array}
\right)
\end{align*}
\begin{align*}
A_3 = \left(
\begin{array}{ccc}
-0.8   & 1.6  &  1.2\\
1.6  &  -1.2  & -1.2\\
1.6  & -2.4 &   1.2
\end{array}
\right),\quad B_3 = \left(
\begin{array}{c}
1\\
1\\
2
\end{array}
\right).
\end{align*}
This system has all Lyapunov exponents positive given by $\lambda_1 = 0.4578$, $\lambda_2 = 0.1191$, and $\lambda_3 = 0.0544$. Then, from Theorem \ref{ctrl_gain_thm} we have the critical probability as
\[p^* = 1 - \frac{1}{e^{2(\sum_{i}\lambda_i)}} = 0.7170.\]

\begin{figure}[ht!]
\centering
\includegraphics[width=0.46\textwidth]{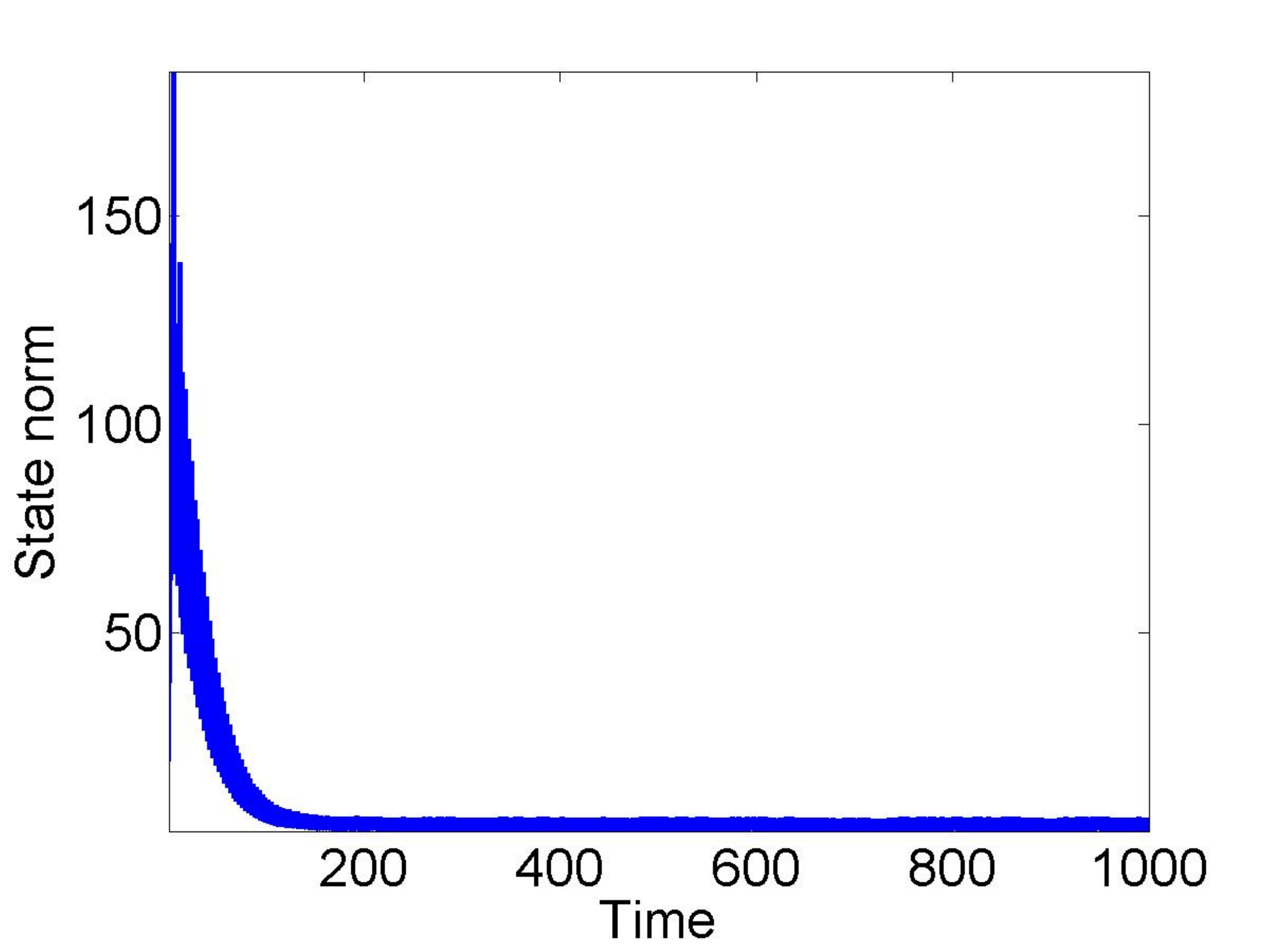}
\caption{State norm for non-erasure probability above $p^* < p=0.8170$.}
\label{fig:fig5}
\end{figure}

\begin{figure}[ht!]
\centering
\includegraphics[width=0.46\textwidth]{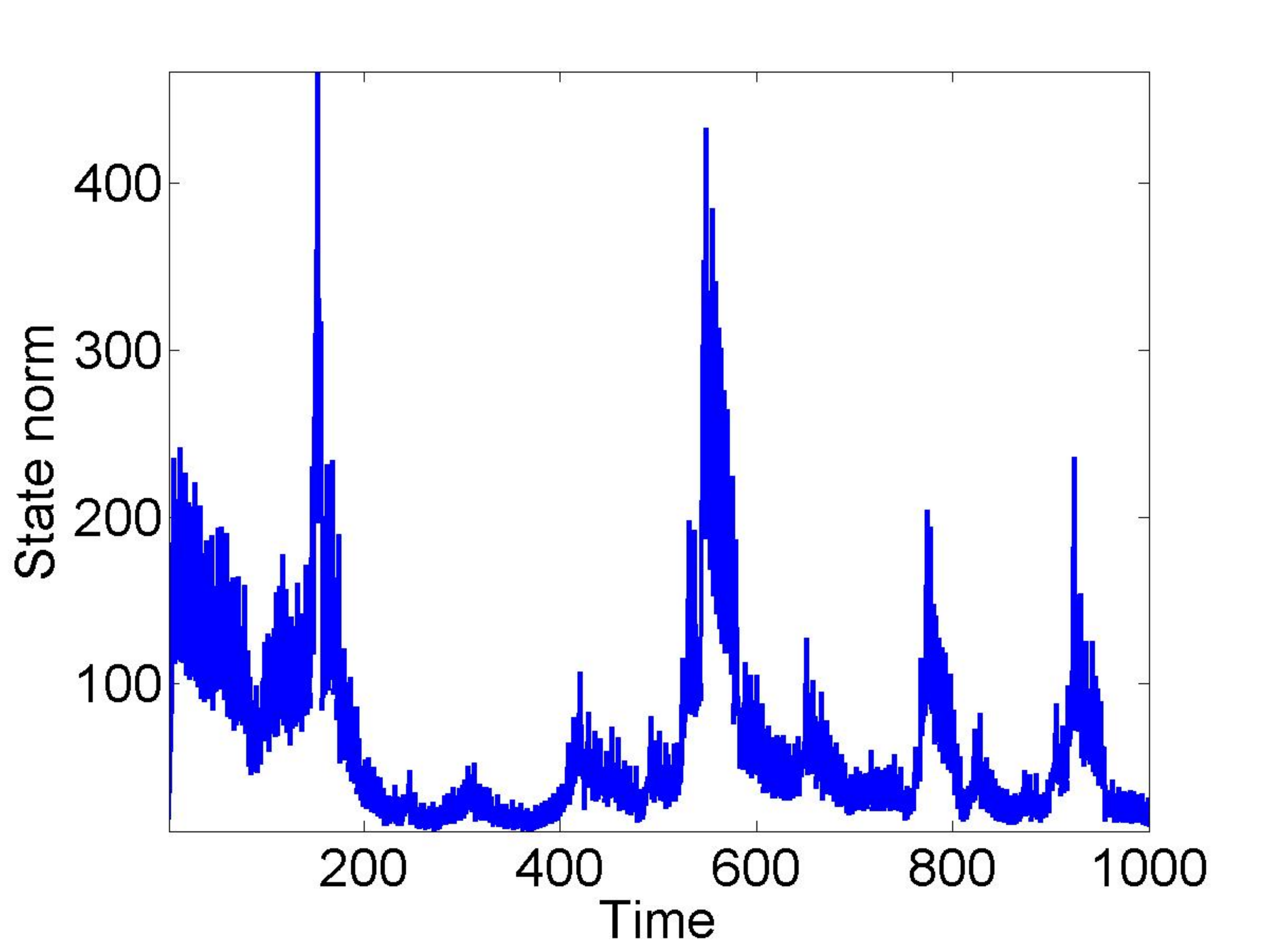}
\caption{State norm for non-erasure probability below $p^* > p=0.6170$.}
\label{fig:fig7}
\end{figure}

We add to the system some white zero mean Gaussian noise with variance $\sigma_G^2 = 0.01$. Now, we plot the norm of the state for the case with uncertainty in control for two values of the non-erasure probability, $p = 0.8170 > p^*$ and $p = 0.6170 < p^*$. Furthermore, in case of uncertainty in control the norm has been averaged over $1000$ realizations of the actuation uncertainty sequence. We clearly see, in Fig. (\ref{fig:fig5}), the norm of the state above the critical probability, $p = 0.8170 > p^*$, stabilizes close to zero, due to the addition of the Gaussian noise vector. In the case of the probability of non-erasure $p = 0.6170 < p^*$ less than the critical probability, in Fig. (\ref{fig:fig7}), the state norm fluctuates significantly as compared to the case for $p = 0.8170 > p^*$. This indicates the system is fragile to the sequence of uncertainties below the critical probability.

\section{Conclusions}
\label{inf}
In this paper, we have studied the problem of control over uncertain channels between the plant and the controller for an LTV system. The results provided necessary and sufficient conditions for the feedback control system to be mean square exponentially stable. We provide computable necessary condition for the $M<N$ input case, where $N$ is the dimension of the state space. For the $N$-input case, we give a computable necessary condition, that is also shown to be sufficient. The necessary conditions are expressed in terms of the mean and variance of the stochastic channel uncertainty and the instability of the open-loop dynamics, as captured by the positive Lyapunov exponents of the open-loop system. The results in this paper generalize the existing results known in the case of LTI systems and Lyapunov exponents emerge as the natural generalization of eigenvalues from LTI systems to LTV systems. Simulation results verify the main conclusion for the single input case for a special case of an erasure channel. While the result provides a necessary condition, our simulation results indicate this condition may also be sufficient. The proof technique presented in this paper can be extended to prove limitation results for the estimation of LTV systems over erasure channels \cite{vaidya_amit_journal}.

\section{Acknowledgment}
Financial support from National Science Foundation grant ECCS 1002053 for this work is greatly acknowledged.

\bibliographystyle{IEEETran}       % Include this if you use bibtex
\bibliography{ref,ref1,ref2}           % and a bib file to produce the
\end{document}